\documentclass[12pt]{article}

\usepackage{amsmath, amsthm, amssymb,bm,url}

\numberwithin{equation}{section}

\newcommand{\ABS}[1]{\left(#1\right)} 

\newcommand{\mf}[1]{\mathfrak{#1}}

\DeclareMathOperator{\sgn}{sgn}
\DeclareMathOperator{\pf}{Pf}
\DeclareMathOperator{\hf}{Hf}
\DeclareMathOperator{\var}{Var}

\def\ii{\mathtt{i}}

\def\({ \left( }
\def\){ \right)}

\def\trans#1{\mathord{#1}^{\mathrm{t}}}

\def\inlaw{\stackrel{\mathtt{d}}{=}}
\def\la{\lambda}
\def\La{\Lambda}

\theoremstyle{plain}
\newtheorem{theorem}{Theorem}[section]
\theoremstyle{prop}
\newtheorem{proposition}[theorem]{Proposition}
\newtheorem{lemma}[theorem]{Lemma}
\newtheorem{corollary}[theorem]{Corollary}
\theoremstyle{definition}

\newtheorem{remark}[theorem]{Remark}
\theoremstyle{conjecture}

\title{\textbf{Correlation functions for zeros of 
a Gaussian power series and Pfaffians}}
\author{\textsc{Sho Matsumoto and Tomoyuki Shirai}}

\date{\empty}

\begin{document}

\maketitle

\begin{abstract}
We show that the zeros of the random power series 
with i.i.d. real Gaussian coefficients
form a Pfaffian point process. We further show that
the product moments for absolute values and signatures of the power series 
can also be expressed by Pfaffians. 
\end{abstract}



\section{Introduction}

Zeros of Gaussian processes have attracted much attention 
for many years both from theoretical and practical points of view. 
The first significant contribution to this study was made 
by Paley and Wiener \cite{PW}.   
They computed the expectation of the number of zeros of 
(translation invariant) analytic Gaussian processes on a strip in the
complex plane, which
are defined as Wiener integrals. 
Their work was motivated by the theory, developed by Bohr and Jessen, 
of almost periodic
functions in the complex domain arising from Riemann's zeta function.
Kac 
gave an explicit expression for the probability density function 
of real zeros of a random polynomial 
\[
 f_n(z) = \sum_{k=0}^n a_k z^k
\] 
with i.i.d. real standard Gaussian coefficients $\{a_k\}_{k=0}^n$ 
and obtains precise asymptotics of the numbers of real zeros 
as $n \to \infty$ \cite{Kac}. 
Rice also obtained similar formulas for the zeros of 
random Fourier series with Gaussian coefficients in 
the theory of filtering \cite{Rice}. 
Their results have been extended in various ways (e.g. \cite{EK,LS,SV}) 
and generalizations of their formulas are sometimes called the Kac-Rice formulas. 
A recent remarkable result on zeros of Gaussian processes 
is that the complex Gaussian process 
$f_{\mathbb{C}}(z) := \sum_{k=0}^{\infty} \zeta_k z^k$ with i.i.d. 
complex standard Gaussian coefficients form 
a determinantal point process on the open unit disk $\mathbb{D}$ 
associated with the Bergman kernel $K(z,w) = \frac{1}{(1-z\bar{w})^2}$, 
which was found by Peres and Vir\'ag \cite{PV}. 
Krishnapur 
extended this result to the zeros of the determinant of the power series with 
coefficients being i.i.d. Ginibre matrices \cite{Kri}. 

In the present paper, we deal with the Gaussian power series
\begin{equation} \label{eq:definition-f}
f(z)= \sum_{k=0}^\infty a_k z^k,
\end{equation}
where $\{a_k\}_{k=0}^\infty$ are i.i.d. {\it real} standard 
Gaussian random variables.
The radius of convergence of $f$ is almost surely $1$, 
and the set of the zeros of $f$ forms a point process on the open unit disc
$\mathbb{D}$ as does that of $f_{\mathbb{C}}$. 
The primary difference between $f$ and $f_{\mathbb{C}}$ comes from the fact that 
$f(z)$ is a real Gaussian process 
when the parameter $z$ is restricted on $(-1,1)$ and each realization of 
$f(z)$ has symmetry with respect to the complex conjugation so that
there appear both real zeros and complex ones in conjugate pairs. 

Our main purpose is to 
show that both correlation functions for real zeros and complex zeros of
$f$ are given by Pfaffians, i.e., 
they form Pfaffian point processes on $(-1,1)$ and $\mathbb{D}$, respectively. 
The most known examples of Pfaffian point processes appeared as 
random eigenvalues of the Gaussian orthogonal/symplectic ensembles. 
Real and complex eigenvalues of the \emph{real} Ginibre ensemble 
are also proved to be Pfaffian point processes on $\mathbb{R}$ and $\mathbb{C}$,
respectively \cite{BS, FN}. 
Recently, it is shown that 
the particle positions of instantly coalescing (or annihilating) 
Brownian motions on the real line under the maximal entrance law 
form a Pfaffian point process on $\mathbb{R}$ \cite{TZ}, 
which is closely related to the real Ginibre ensemble. 
Our result on correlation functions of zeros of $f$ is added to the list of Pfaffian
point processes, which is also obtained independently in \cite{F} via random matrix theory.
Here we will give a direct proof by using Hammersley's formula for 
correlation functions of zeros of Gaussian analytic functions 
and a Pfaffian-Hafnian identity due to Ishikawa, Kawamuko, and Okada \cite{IKO}. 
This is a similar way to that which was taken in \cite{PV} 
to prove that the zeros of $f_{\mathbb{C}}$ form a determinantal
point process, and in the process of our calculus for real zero correlations, 
we obtain new Pfaffian formulas for a real Gaussian process. 
The family $\{f(t)\}_{-1<t<1}$ can be regarded as
a centered real Gaussian process with covariance kernel $(1-st)^{-1}$. 
We show that, for any $-1<t_1,t_2,\dots,t_n<1$,
both the moments of absolute values 
$E[|f(t_1) f(t_2) \cdots f(t_n)|]$ and 
those of signatures $E[\sgn (f(t_1)) \cdots \sgn (f(t_n))]$
are also given by Pfaffians. 
We stress that it should be \emph{surprising} because
such combinatorial formulas cannot be expected 
for 
general centered Gaussian processes. 
These are special features for the Gaussian process 
with covariance kernel $(1-st)^{-1}$.

The paper is organized as follows. 
In Section 2, we state our main results for correlations of 
real and complex zeros of $f$ 
(Theorems~\ref{thm:correlation} and \ref{thm:c-correlation}), and 
we give new product moment formulas 
for absolute values and signatures of $f$ 
(Theorems~\ref{thm:absolute} and \ref{thm:sgn}). 
Also we observe a negative correlation property of real and complex zeros 
by showing negative correlation inequalities for $2$-correlation 
functions. The asymptotics of the number of real zeros inside 
intervals growing to $(-1,1)$ is also shown. 
In Section 3, we recall the well-known Cauchy determinant formula 
and the Wick formula for product moments of Gaussian random variables. 
In Section 4, after we show an identity in law for $f$ and $f'$ 
given that $f$ is vanishing at some points, 
we give a preliminary version of Pfaffian formulas 
(Proposition~\ref{prop:differential}) 
for the derivative of the expectation of products of sign functions. 
In Sections 5, 6 and 7, we give the proofs of our results stated in 
Section 2.

\section{Results}

\subsection{Pfaffians}

Our main results will be described by using Pfaffians, so
let us recall the definition.
For a $2n \times 2n$ skew symmetric matrix
 $B=(b_{ij})_{i,j=1}^{2n}$,
the Pfaffian of $B$ is defined by
\[
\pf (B)= \sum_{\eta} \epsilon(\eta) b_{\eta(1) \eta(2)}
b_{\eta(3) \eta(4)} \cdots b_{\eta(2n-1) \eta(2n)},
\]
summed over all permutations 
$\eta$ on $\{1,2,\dots, 2n\}$
satisfying $\eta(2i-1)<\eta(2i)$ $(i=1,2,\dots,n)$ and 
$\eta(1)< \eta(3)< \cdots <\eta(2n-1)$.
Here $\epsilon(\eta)$ is the signature of $\eta$.
For example, 
\begin{equation}
\pf (B)=b_{11} \quad \text{if $n=1$} \qquad \text{and}
\qquad 
\pf (B)=b_{12}b_{34}-b_{13}b_{24}+b_{14}b_{23}
\quad \text{if $n=2$}.
\label{pfaffian} 
\end{equation}

For an upper-triangular array $A=(a_{ij})_{1 \le i<j \le 2n}$,
we define the Pfaffian of $A$ as that of the skew-symmetric matrix 
$B=(b_{ij})_{i,j=1}^{2n}$, each entry of which is 
$b_{ij}=-b_{ji}=a_{ij}$ if $i<j$ and $b_{ii}=0$.

\subsection{Notation}

 We will often use the following functions: 
for $-1 <s,t<1$, 
\begin{align}
\sigma(s,t)=&\frac{1}{1-st}, \qquad \qquad \mu(s,t)= \frac{s-t}{1-st}, 
\label{eq:definition-sm}\\
c (s,t) 
=& \frac{\sigma(s,t)}{\sqrt{\sigma(s,s) \sigma(t,t)}}
=\frac{\sqrt{(1-s^2)(1-t^2)}}{1- s t},
\label{eq:definition-c}
\end{align} 
where $\sigma(s,t)$ is the covariance function for the 
real Gaussian process $\{f(t)\}_{-1<t<1}$ and
$c(s,t)$ is the correlation coefficient between $f(s)$ and $f(t)$. 
We define the skew symmetric matrix kernel $\mathbb{K}$ by
\[
\mathbb{K} (s,t)= \begin{pmatrix}
\mathbb{K}_{11}(s,t) & \mathbb{K}_{12}(s,t) \\
\mathbb{K}_{21}(s,t) & \mathbb{K}_{22}(s,t)  
\end{pmatrix}
\]
with
\begin{align*}
\mathbb{K}_{11}(s,t)=& \frac{s-t}{\sqrt{(1-s^2)(1-t^2)} (1-st)^2}, &
\mathbb{K}_{12}(s,t) =& \sqrt{\frac{1-t^2}{1-s^2}} \frac{1}{1-st}, \\
\mathbb{K}_{21}(s,t) =& -\sqrt{\frac{1-s^2}{1-t^2}} \frac{1}{1-st}, 
&
\mathbb{K}_{22}(s,t)=& 
\sgn(t-s) \arcsin c(s,t), 
\end{align*}
where $\sgn(t)  = |t|/t$ for $t \not= 0$ and $\sgn (t) =0$ for $t=0$. 
Note that $\mathbb{K}_{12}(s,t) = -\mathbb{K}_{21}(t,s)$
and 
\begin{equation}
\mathbb{K}(s,t) = 
 \begin{pmatrix} \frac{\partial^2}{\partial s \partial t} \mathbb{K}_{22}(s,t) &
 \frac{\partial}{\partial s} \mathbb{K}_{22}(s,t) \\
 \frac{\partial}{\partial t} \mathbb{K}_{22}(s,t) &
\mathbb{K}_{22}(s,t)
 \end{pmatrix}.
\label{derivative}
\end{equation}

For $-1<t_1,t_2,\dots,t_n<1$, we 
write
$ (\mathbb{K}(t_i,t_j))_{i,j=1}^n$ 
for the $2n \times 2n$ skew symmetric matrix
\[
\begin{pmatrix}
\mathbb{K}(t_1,t_1) & \mathbb{K}(t_1,t_2) & \ldots & \mathbb{K}(t_1,t_n) \\
\mathbb{K}(t_2,t_1) & \mathbb{K}(t_2,t_2) & \ldots & \mathbb{K}(t_2,t_n) \\
\vdots & \vdots & \ddots & \vdots \\
\mathbb{K}(t_n,t_1) & \mathbb{K}(t_n,t_2) & \ldots & \mathbb{K}(t_n,t_n) 
\end{pmatrix},
\]
and denote the covariance matrix of the 
real Gaussian vector $(f(t_1),f(t_2),\dots,f(t_n))$ by 
\begin{equation} \label{eq:definition-Sigma}
\Sigma(\bm{t})=\Sigma(t_1,\dots,t_n)=(\sigma(t_i,t_j))_{i,j=1}^n.
\end{equation}

Throughout this paper, $\mf{X}_n$ denotes 
the set of all sequences $\bm{t}=(t_1,\dots,t_n)$
of $n$ distinct real numbers in the interval $(-1,1)$.
If $\bm{t} \in \mf{X}_n$ then $\Sigma(\bm{t})$ 
is positive-definite.

\subsection{Real zero correlations}

Our first theorem states that
the real zero distribution of $f$ 
defined in \eqref{eq:definition-f}
forms a Pfaffian point process.

\begin{theorem} \label{thm:correlation}
Let $\rho_n$ be the $n$-point correlation function
of real zeros of $f$.
Then  
\[
\rho_n(t_1,\dots,t_n)= \pi^{-n}
 \pf (\mathbb{K}(t_i,t_j))_{i,j=1}^n \qquad
(-1< t_1,\dots,t_n <1).
\]
\end{theorem}

For example,  
the first two correlations are given as follows: 
 \begin{align}
\rho_1(s) &= \pi^{-1} \mathbb{K}_{12}(s,s), \nonumber\\
 \rho_2(s,t) &= \pi^{-2}\{\mathbb{K}_{12}(s,s) \mathbb{K}_{12}(t,t)
- \mathbb{K}_{11}(s,t) \mathbb{K}_{22}(s,t)
+ \mathbb{K}_{12}(s,t) \mathbb{K}_{21}(s,t)\}, 
\label{2correlation}
\end{align}
from which we easily see that 
\[
 \rho_1(s) = \frac{1}{\pi(1-s^2)}, \quad 
\rho_2(s,t) = \frac{1}{2\pi (1-s^2)^3} |t-s| + O(|t-s|^2) 
\]
as $t \to s$. 
The first correlation is observed by Kac and many others 
although Kac considered the random polynomial with i.i.d. 
real Gaussian coefficients. 
The second asymptotic expression means that 
the real zeros of $f$  repel 
each other as expected. 
Moreover, we can show that the 2-correlation is 
negatively correlated. 

\begin{corollary}\label{2point}
Let $R(s,t)= \frac{\rho_2(s,t)}{\rho_1(s)\rho_1(t)}$ be the normalized 
$2$-point correlation function. Then, 
$R(s,s)=0$, $R(s,\pm 1)=1$ and 
$R(s,t)$ is strictly increasing (resp. decreasing) for 
$t \in [s,1]$ (resp. $t \in [-1,s]$). 
In particular, $\rho_2(s,t) \le \rho_1(s) \rho_1(t)$ for every 
$s, t \in (-1,1)$. 
\end{corollary}
By using (\ref{2correlation}), 
we can also compute the mean and 
variance of the number of points inside $[-r,r]$. 

\begin{corollary} \label{cor:variance}
Let $N_r$ be the number of real zeros in the interval $[-r,r]$ for 
$0<r<1$. Then, 
\begin{align*}
 E N_r = \frac{1}{\pi} \log \frac{1+r}{1-r}, \qquad 
 \var N_r = 2\Big(1 - \frac{2}{\pi}\Big) E N_r + O(1)
\end{align*}
as $r \to 1$. 
\end{corollary}

\begin{remark}
The kernel $\mathbb{K}$ in Theorem \ref{thm:correlation} is not determined uniquely.
For example, 
we can replace $\mathbb{K}$  by 
$\mathbb{K}'$, which is defined by
\begin{align*}
\mathbb{K}'_{11}(s,t)=&\frac{s-t}{(1-s t)^2}, \qquad\qquad
\mathbb{K}'_{12}(s,t)=-\mathbb{K}_{21}'(t,s)= \frac{1}{1-s t}, \\
\mathbb{K}'_{22}(s,t)=& \frac{\sgn(t-s)}{\sqrt{(1-s^2)(1-t^2)}}
\arcsin c(s,t).
\end{align*}
In fact, 
if we set 
\[
Q(s,t) = \delta_{st} \begin{pmatrix} \sqrt{1-t^2} & 0 \\ 0 & 
\frac{1}{\sqrt{1-t^2}} \end{pmatrix}
\]
then 
$(Q(t_i,t_j))_{i,j=1}^n  \cdot (\mathbb{K}(t_i,t_j))_{i,j=1}^n \cdot  
(Q(t_i,t_j))_{i,j=1}^n = (\mathbb{K}'(t_i,t_j))_{i,j=1}^n$,
and 
therefore two Pfaffians associated with $\mathbb{K}$ and $\mathbb{K}'$ coincide
from the following well-known identity:
for any $2n \times 2n$ matrix $A$ and 
$2n \times 2n$ skew symmetric matrix $B$,
 $\pf (A B \trans{A})=(\det A)(\pf B)$.
\end{remark}

\subsection{Pfaffian formulas for a real Gaussian process}

As corollaries of the proof of Theorem \ref{thm:correlation},
we obtain Pfaffian expressions for 
averages of $|f(t_1) \cdots f(t_n)|$ and 
$\sgn f(t_1) \cdots \sgn f(t_n)$.

\begin{theorem} \label{thm:absolute}
For $\bm{t}=(t_1,\dots,t_n) \in \mf{X}_n$, we have
\[
E[ |f(t_1) f(t_2)\cdots f(t_n)| ]
= \ABS{\frac{2}{\pi}}^{n/2} (\det \Sigma(\bm{t}))^{-\frac{1}{2}}
\pf (\mathbb{K}(t_i,t_j))_{i,j=1}^n.
\]
\end{theorem}

\begin{theorem} \label{thm:sgn}
For $(t_1,\dots,t_{2n}) \in \mf{X}_{2n}$,
we have
\begin{align}
\lefteqn{E[ \sgn f(t_1) \sgn f(t_2) \cdots \sgn f(t_{2n}) ]} \notag \\
&= \ABS{\frac{2}{\pi}}^{n} 
\prod_{1 \le i<j \le 2n} \sgn (t_j-t_i) \cdot
\pf \ABS{ \mathbb{K}_{22}(t_i,t_j)}_{i,j=1}^{2n}. 
\label{eq:sgn-pfaffian-formula0}
\end{align}
In particular, if $-1<t_1< t_2< \cdots <t_{2n} <1$, then
\begin{equation}
E[ \sgn f(t_1) \sgn f(t_2) \cdots \sgn f(t_{2n}) ]
=\ABS{\frac{2}{\pi}}^{n} 
\pf \ABS{ \arcsin c(t_i,t_j) }_{1 \le i<j \le 2n},
\label{eq:sgn-pfaffian-formula} 
\end{equation}
where $c(s,t)$ is defined in \eqref{eq:definition-c}.
\end{theorem}
 
We can easily see that
$E[\sgn f(t_1) \cdots \sgn f(t_n)]=0$ when $n$ is odd.
More generally, 
if $(X_1,\dots,X_n)$ is a centered real Gaussian vector, then
$E[\sgn X_1 \cdots \sgn X_n]=0$ for $n$ odd.
The formula \eqref{eq:sgn-pfaffian-formula0} with $n=1$ says
that $E[\sgn f(s) \sgn f(t)]= \frac{2}{\pi} \arcsin c(s,t)$, and hence
\eqref{eq:sgn-pfaffian-formula} can be rewritten
as 
\[
E[ \sgn f(t_1) \sgn f(t_2) \cdots \sgn f(t_{2n}) ]
= \pf \ABS{ E[\sgn f(t_i) \sgn f(t_j)] }_{1 \le i<j \le 2n}.
\]
As explained later, the Wick formula \eqref{wick} 
provides us a similar formula for products of real Gaussian random
variables, however, 
such neat formulas for $E[|X_1 X_2 \cdots X_n|]$ are not known 
for general $n$ except the cases with $n=2,3$ (\cite{Nabeya1,Nabeya2}).
Similarly, there is no known formula for 
$E[\sgn X_1 \sgn X_2 \cdots \sgn X_{2n}]$
except the $n=1$ case 
\[
E[\sgn X_1 \sgn X_2] = \frac{2}{\pi} \arcsin
\frac{\sigma_{12}}{\sqrt{\sigma_{11} \sigma_{22}}}, 
\]
where $\sigma_{ij} = E[X_i X_j]$ for $i,j=1,2$. 
Theorem \ref{thm:absolute} and Theorem \ref{thm:sgn} state that
the moments $E[|X_1 \cdots X_n|]$ and $E[\sgn X_1 \cdots \sgn X_n]$
have Pfaffian expressions if the covariance matrix of 
the real Gaussian vector $(X_1,\dots,X_n)$
is of the form $((1-t_i t_j)^{-1})_{i,j=1}^n$.

\subsection{Complex zero correlations}

The complex zero distribution of $f$ also forms
a Pfaffian point process.
Put
$\mathbb{D}_+=\{z \in \mathbb{C} \ | \ |z|<1, \ \Im z>0\}$,
the upper half of the open unit disc.
We write $\ii=\sqrt{-1}$.

\begin{theorem}\label{thm:c-correlation}
Let $\rho^{\mathrm{c}}_n$ be the $n$-point correlation function
for complex zeros of $f$.
For $z_1, z_2,\dots, z_n \in \mathbb{D}_+$, 
\[
 \rho^{\mathrm{c}}_n(z_1,\dots,z_n) = 
\frac{1}{(\pi \ii)^n} \prod_{j=1}^n \frac{1}{|1-z_j^2|} 
\cdot \pf(\mathbb{K}^{\mathrm{c}}(z_i,z_j))_{i,j=1}^n, 
\]
where $\mathbb{K}^{\mathrm{c}}(z,w)$ is the $2 \times 2$ matrix kernel 
\def\zbar{\bar{z}}
\def\wbar{\bar{w}}
\[
\mathbb{K}^{\mathrm{c}}(z,w) = 
\begin{pmatrix}
 \frac{z-w}{(1-zw)^2} &  \frac{z-\bar{w}}{(1-z \bar{w})^2} \\
\frac{\bar{z}-w}{(1-\bar{z} w)^2} &  \frac{\bar{z}-\bar{w}}{(1-\bar{z} \bar{w})^2} 
\end{pmatrix}.
\]
\end{theorem}

For example, the first two correlations are given by
\begin{align*}
\rho_1^{\mathrm{c}}(z)=& \frac{|z-\overline{z}|}{\pi |1-z^2| (1-|z|^2)^2}, \\
\rho_2^{\mathrm{c}}(z,w)
=& \rho_1^{\mathrm{c}}(z)\rho_1^{\mathrm{c}}(w)
+ \frac{1}{\pi^2|1-z^2||1-w^2|} 
\ABS{\left| \frac{z-w}{1-zw}\right|^2 
- \left| \frac{z-\bar{w}}{1-z\bar{w}}\right|^2}. 
\end{align*}
It is easy to verify that $\rho_2^{\mathrm{c}}(z,w) < 
\rho_1^{\mathrm{c}}(z)\rho_1^{\mathrm{c}}(w)$ for $z,w \in
\mathbb{D}_+$, which implies negative correlation as well as the case of
real zeros. 

As we mentioned, Theorem \ref{thm:correlation} and Theorem \ref{thm:c-correlation}
are obtained independently in \cite{F} via random matrix theory,
but Theorem \ref{thm:absolute} and Theorem \ref{thm:sgn} are new.

\section{Cauchy's determinants and Wick formula}

In this short section, we review Cauchy's determinants and 
the Wick formula, which are essential throughout this paper.

\subsection{Cauchy's determinant and its variations}

The following identity for a determinant,
the so-called Cauchy determinant identity,
is well known in combinatorics, see, e.g., \cite[Proposition 4.2.3]{CST}.
\begin{equation}  \label{eq:Cauchy}
\det \ABS{ \frac{1}{1-x_i y_j}}_{i,j=1}^n
= \frac{\prod_{1 \le i<j \le n} (x_i-x_j)(y_i-y_j)}
{\prod_{i=1}^n \prod_{j=1}^n (1-x_i y_j)}.
\end{equation}
Here the $x_i$, $y_j$ are formal variables, but
we will assume that they are complex numbers 
with absolute values smaller than $1$
when we apply formulas contained in this subsection.
For each $i=1,2,\dots,n$, we 
define $q_i(\bm{x})=q_i(x_1,\dots,x_n)$ by
\begin{equation} \label{eq:definition-q}
q_i(\bm{x})= 
\frac{1}{1-x_i^2} \prod_{\begin{subarray}{c}  1 \le k \le n \\ k \not=i \end{subarray}}
\frac{x_i-x_k}{1-x_i x_k}.
\end{equation}
Using \eqref{eq:Cauchy},
we have
\begin{equation}  \label{eq:product-of-q}
q_1(\bm{x}) q_2(\bm{x}) \cdots q_n (\bm{x}) 
=(-1)^{n(n-1)/2} \det \ABS{ \frac{1}{1-x_i x_j}}_{i,j=1}^n
\end{equation}

Recall the definition of Hafnians, which are sign-less analogs of Pfaffians.
For a $2n \times 2n$ symmetric matrix $A=(a_{ij})_{i,j=1}^{2n}$, 
the Hafnian of $A$ is defined by 
\begin{equation}
\hf A = \sum_{\eta} a_{\eta(1) \eta(2)} a_{\eta(3) \eta(4)} \cdots a_{\eta(2n-1) \eta(2n)},
\label{hafnian}
\end{equation}
summed over all permutations 
$\eta$ on $\{1,2,\dots, 2n\}$
satisfying $\eta(2i-1)<\eta(2i)$ $(i=1,2,\dots,n)$ and 
$\eta(1)< \eta(3)< \cdots <\eta(2n-1)$.

A Pfaffian version of Cauchy's determinant identity is Schur's Pfaffian identity
(see, e.g., \cite{IKO}):
\[
\pf \ABS{ \frac{x_i-x_j}{1-x_i x_j}}_{i,j=1}^{2n}=\prod_{1 \le i<j \le 2n} \frac{x_i-x_j}{1-x_ix_j}.
\]
The following formula due to 
Ishikawa, Kawamuko, and Okada \cite{IKO}
will be an important factor in our proofs of theorems. 
\begin{equation}
\prod_{1 \le i<j \le 2n} \frac{x_i-x_j}{1-x_ix_j} 
\cdot \hf \ABS{ \frac{1}{1-x_i x_j} }_{i,j=1}^{2n}
= \pf \ABS{ \frac{x_i-x_j}{(1-x_ix_j)^2} }_{i,j=1}^{2n}.
\label{eq:pfaffian-hafnian}
\end{equation}

\subsection{Wick formula}

We recall the method for computations of
expectations of polynomials in real Gaussian random variables. 
Let $Y_1,\dots,Y_{n}$ be 
a centered real Gaussian random vector.
Then $E[Y_1 \cdots Y_n]=0$ if $n$ is odd, and 
\begin{equation}
E[Y_1 Y_2 \cdots Y_{n}]= \hf (E[Y_i Y_j])_{i,j=1}^n
\label{wick}
\end{equation}
if $n$ is even.
For example,
$E[Y_1 Y_2 Y_3 Y_4] = E[Y_1 Y_2] E [Y_3 Y_4]+
E[Y_1 Y_3] E [Y_2 Y_4]+ E[Y_1 Y_4] E [Y_2 Y_3]$.
See, e.g., survey \cite{Z} for details.

\section{Derivatives of sign moments}

In this section, we provide
a preliminary version of Pfaffian formulas for Theorem~\ref{thm:sgn}.

\subsection{Derivatives of real Gaussian processes}

The derivative of the product-moment of signs of 
a smooth real Gaussian process
is given by a conditional expectation in the following way.

\begin{lemma} \label{lem:GaussianProcess}
Let $\{X(t)\}_{-1<t<1}$ be a smooth real Gaussian process
with covariance kernel $K$.
Let $(t_1,t_2,\dots,t_n,s_1,s_2,\dots,s_m) \in \mf{X}_{n+m}$
and suppose that $\det K(\bm{t})= \det (K(t_i, t_j))_{i,j=1}^n$
does not vanish.
Then,
\begin{align*}
&\frac{\partial^{n}}{\partial t_1 \partial t_2 \cdots
\partial t_{n}}
E[ \sgn X(t_1) \cdots \sgn X(t_{n}) 
\sgn X(s_1) \cdots \sgn X(s_{m})] \\
=& \ABS{\frac{2}{\pi}}^{\frac{n}{2}} (\det K(\bm{t}))^{-\frac{1}{2}}
E[X'(t_1) \cdots X'(t_n)\sgn X(s_1) \cdots \sgn X(s_{m})
\ | \ X(t_1)= \cdots =X(t_n)=0].
\end{align*}
\end{lemma}

\begin{proof}
We will give a heuristic proof.
The derivative of $\sgn t$ is 
$\frac{\partial}{\partial t} \sgn t= 2\delta_0(t)$,
where $\delta_0(t)$ is Dirac's delta function at $0$.
Hence, if we abbreviate as 
$Y(\bm{s})= \sgn X(s_1) \cdots \sgn X(s_m)$, then
\begin{align*}
\lefteqn{\frac{\partial^{n}}{\partial t_1 \partial t_2 \cdots
\partial t_{n}}
E[ \sgn X(t_1) \cdots \sgn X(t_{n}) \cdot Y(\bm{s})]}  \\
&= 2^n E[\delta_0(X(t_1)) X'(t_1) \cdots \delta_0(X(t_n)) X'(t_n)
\cdot Y(\bm{s})] \\
&= 2^n E[ X'(t_1) \cdots X'(t_n) \cdot Y(\bm{s}) \ | \ 
X(t_1)= \cdots =X(t_n)=0] \cdot p_{\bm{t}}(\bm{0}),
\end{align*}
where $p_{\bm{t}}(\bm{0})$ is the density 
of the Gaussian vector $(X(t_1),\dots,X(t_n))$ at $(0,\dots,0)$.
Since $p_{\bm{t}}(\bm{0})=(2\pi)^{-n/2} (\det K(\bm{t}))^{-1/2}$,
the claim follows.

The above formal computation can be justified by using 
Watanabe's generalized Wiener functionals in the framework of Malliavin
 calculus over abstract Wiener spaces \cite{W, W-AP}. 
\end{proof}

\subsection{Conditional expectations}

Recall the Gaussian power series $f$
defined in \eqref{eq:definition-f}
and functions $\sigma(s,t)$ and $\mu(s,t)$ defined in \eqref{eq:definition-sm}.
The process $\{f(t)\}_{-1 <t<1}$ is 
centered real Gaussian with covariance kernel $\sigma(s,t)$.
The following identity in law is a crucial property 
which the Gaussian process with covariance kernel $\sigma(s,t)$ enjoys. 

\begin{lemma} \label{lem:conditional}
For given $(t_1,t_2,\dots, t_n) \in \mf{X}_n$, we have 
\begin{equation}
(f | f(t_1)= \cdots = f(t_n) =0) 
\inlaw \mu(\cdot, \bm{t}) f,
\label{inlaw1}
\end{equation}
where $\mu(s, \bm{t}) = \prod_{i=1}^n \mu(s,t_i)$. 
Moreover, 
$$
(f, f'(t_i), i=1,\dots,n | f(t_1)= \cdots = f(t_n) =0) 
\inlaw (\mu(\cdot, \bm{t}) f, q_i(\bm{t}) f(t_i), i=1,\dots, n), 
$$
where $q_i(\bm{t})=q_i(t_1,\dots,t_n)$ is defined in \eqref{eq:definition-q}.
\end{lemma}

\begin{proof}
If a Gaussian process $X$ has a covariance kernel $K(x,y)$, then 
that of the Gaussian process $(X | X(t)=0)$ (i.e. $X$ given $X(t)=0$)
is equal to $K(x,y) - K(x,t)K(t,y)/K(t,t)$ 
whenever $K(t,t)>0$. 
In the case of the kernel $\sigma(x,y)$, we see that 
\[
 \sigma(x,y) - \frac{\sigma(x,t) \sigma(t,y)}{\sigma(t,t)} 
= \mu(x,t)\mu(y,t) \sigma(x,y). 
\]
This implies that $(f | f(t)=0) \inlaw \mu(\cdot, t) f$ as 
a process. Hence we obtain (\ref{inlaw1}) by induction. 
As $f'$ is a linear functional of $f$, 
we also have the identity in law as a Gaussian system 
\begin{align*}
(f, f' \ | \ f(t_1)= \cdots = f(t_n) =0)
&\inlaw 
(\mu(\cdot, \bm{t}) f, \mu'(\cdot, \bm{t}) f 
+ \mu(\cdot, \bm{t}) f').  
\end{align*}
Since $\mu(t_i, \bm{t}) = 0$ and $\mu'(t_i, \bm{t}) = q_i(\bm{t})$ for 
every $i=1,2,\dots,n$, 
we obtain the second equality in law. 
\end{proof}

\subsection{Pfaffian expressions for derivatives of signs}

The following lemma is a consequence of Lemmas \ref{lem:GaussianProcess} and \ref{lem:conditional}.

\begin{lemma} \label{lem:key-lem}
Let $(\bm{t},\bm{s})=(t_1,t_2,\dots,t_n,s_1,s_2,\dots,s_m) \in \mf{X}_{n+m}$.
Then 
\begin{align*}
\lefteqn{\frac{\partial^{n}}{\partial t_1 \partial t_2 \cdots
\partial t_{n}}
E[ \sgn f(t_1) \cdots \sgn f(t_{n}) 
\sgn f(s_1) \cdots \sgn f(s_{m})]} \\
&=  (-1)^{n(n-1)/2} \prod_{i=1}^n \prod_{j=1}^m \sgn (s_j-t_i) \\
& \qquad \times 
\ABS{\frac{2}{\pi}}^{n/2} 
(\det \Sigma(\bm{t}))^{1/2} E[f(t_1) \cdots f(t_n) 
\sgn f(s_1) \cdots \sgn f(s_m)]
\end{align*}
with $\Sigma(\bm{t})$ defined in \eqref{eq:definition-Sigma}.
\end{lemma}

\begin{proof}
From Lemma \ref{lem:conditional} and \eqref{eq:product-of-q} we have
\begin{align*}
\lefteqn{E[f'(t_1) \cdots f'(t_n) \sgn f(s_1) \cdots \sgn f(s_m) \ | \ 
f(t_1)= \cdots =f(t_n)=0]} \\
&= \prod_{j=1}^m \sgn \mu(s_j, \bm{t}) \cdot 
\prod_{i=1}^n q_i(\bm{t}) \cdot E[f(t_1) \cdots f(t_n) \sgn f(s_1) \cdots 
\sgn f(s_m)] \\
&= \prod_{i=1}^n \prod_{j=1}^m \sgn(s_j-t_i) \cdot (-1)^{n(n-1)/2} 
\det \Sigma(\bm{t}) \cdot E[f(t_1) \cdots f(t_n) \sgn f(s_1) \cdots 
\sgn f(s_m)].
\end{align*}
We have finished the proof by Lemma \ref{lem:GaussianProcess} with $X(t)=f(t)$.
\end{proof}

\begin{proposition}  \label{prop:differential}
For  $\bm{t}=(t_1,\dots, t_{2n}) \in \mf{X}_{2n}$,
we have
\begin{align}
\lefteqn{\frac{\partial^{2n}}{\partial t_1 \partial t_2 \cdots
\partial t_{2n}}
E[ \sgn f(t_1) \sgn f(t_2) \cdots \sgn f(t_{2n}) ]}  
\notag \\
&= \ABS{\frac{2}{\pi}}^{n} 
\prod_{1 \le i<j \le 2n} \sgn (t_j-t_i) \cdot
\prod_{i=1}^{2n} \frac{1}{\sqrt{1-t_i^2}} \cdot 
\pf \ABS{ \frac{t_i-t_j}{(1-t_i t_j)^2}}_{i,j=1}^{2n}.
\label{eq:diffrential-pfaffian}
\end{align}
\end{proposition}

\begin{proof}
Lemma \ref{lem:key-lem} with $m=0$ and with
the replacement $n$ by $2n$  gives
\[
\frac{\partial^{2n}}{\partial t_1  \cdots
\partial t_{2n}}
E[ \sgn f(t_1) \cdots \sgn f(t_{2n}) ] 
=  (-1)^{n} \ABS{\frac{2}{\pi}}^{n} 
(\det \Sigma(\bm{t}))^{1/2} E[f(t_1) \cdots f(t_{2n})].
\]
Here the Wick formula \eqref{wick} gives 
\[
E[f(t_1) \cdots f(t_{2n})] = \hf (E[f(t_i)f(t_j)])_{i,j=1}^{2n}=
\hf \ABS{(1-t_it_j)^{-1}}_{i,j=1}^{2n}
\]
and Cauchy's determinant identity \eqref{eq:Cauchy} gives
\[
(-1)^n (\det \Sigma(\bm{t}))^{1/2}
= \prod_{1 \le i<j \le 2n} \sgn(t_j-t_i) \cdot  \prod_{i=1}^{2n} 
\frac{1}{\sqrt{1-t_i^2}}
\cdot \prod_{1 \le i<j \le 2n} \frac{t_i-t_j}{1-t_it_j}.
\]
Hence we obtain
\begin{align*}
\lefteqn{\frac{\partial^{2n}}{\partial t_1 \partial t_2 \cdots
\partial t_{2n}}
E[ \sgn f(t_1) \sgn f(t_2) \cdots \sgn f(t_{2n}) ]}  
\notag \\
&= \ABS{\frac{2}{\pi}}^{n} 
\prod_{1 \le i<j \le 2n} \sgn (t_j-t_i) \cdot
\prod_{i=1}^{2n} \frac{1}{\sqrt{1-t_i^2}} \cdot 
\prod_{1 \le i<j \le 2n} \frac{t_i-t_j}{1-t_it_j} \cdot 
\hf \ABS{\frac{1}{1-t_it_j}}_{i,j=1}^{2n}.
\end{align*}
The desired Pfaffian expression follows from 
the Pfaffian-Hafnian identity
\eqref{eq:pfaffian-hafnian}.
\end{proof}

\section{Proof of Theorems \ref{thm:absolute} and \ref{thm:sgn}}
\label{sec:absolute-sgn}

\subsection{Some lemmas}

Proposition \ref{prop:differential} can be expressed as
\begin{align*}
\lefteqn{ \frac{\partial^{2n}}{\partial t_1 \partial t_2 \cdots
\partial t_{2n}}
E[ \sgn f(t_1) \sgn f(t_2) \cdots \sgn f(t_{2n}) ]}  \\
&=\ABS{\frac{2}{\pi}}^{n} 
\prod_{1 \le i<j \le 2n} \sgn (t_j-t_i) \cdot
\pf \ABS{ \mathbb{K}_{11}(t_i,t_j)}_{i,j=1}^{2n} \\
&= 
\ABS{\frac{2}{\pi}}^{n} 
\prod_{1 \le i<j \le 2n} \sgn (t_j-t_i) \cdot
\frac{\partial^{2n}}{\partial t_1 \cdots \partial t_{2n}} 
\pf \ABS{ \mathbb{K}_{22}(t_i,t_j)}_{i,j=1}^{2n}.
\end{align*}
If we remove the differential symbol $\frac{\partial^{2n}}{\partial t_1 \cdots \partial t_{2n}}$
in the above equation, then we get the equality in Theorem \ref{thm:sgn}.
The goal of the present subsection is to prove that this observation is veritably true.

For each subset $I$ of $\{1,2,\dots,2n\}$, 
we define the $2n \times 2n$ skew symmetric matrix $\mathbb{L}^I=\mathbb{L}^I(\bm{t})$, 
the $(i,j)$-entry of which is 
\begin{equation}
\mathbb{L}^I_{ij} = 
\begin{cases}
\mathbb{K}_{11}(t_i,t_j)=\frac{\partial^2}{\partial t_i \partial t_j} \mathbb{K}_{22}(t_i,t_j)  &
\text{if $i,j \in I$}, \\
\mathbb{K}_{12}(t_i,t_j)=\frac{\partial}{\partial t_i} 
\mathbb{K}_{22}(t_i,t_j)  &
\text{if $i \in I$ and $j \in I^c$}, \\
\mathbb{K}_{21}(t_i,t_j)=\frac{\partial}{\partial t_j} 
\mathbb{K}_{22}(t_i,t_j)  &
\text{if $i \in I^c$ and $j \in I$}, \\
\mathbb{K}_{22}(t_i,t_j) &
\text{if $i,j \in I^c$}.
\end{cases}
\label{eq:def-L}
\end{equation}
In particular, 
we put $\mathbb{L}^{[k]} = \mathbb{L}^I$ 
if $I=\{1,2,\dots,k\}$ and $\mathbb{L}^{[0]}= \mathbb{L}^\emptyset$.

\begin{lemma} \label{lem:Z-properties}
The following two claims hold true.
\begin{enumerate}
\item For each $k=0,1,\dots,2n-1$, $\frac{\partial}{\partial t_{k+1}} 
\pf \mathbb{L}^{[k]} = \pf \mathbb{L}^{[k+1]}$.
\item $\pf \mathbb{L}^{[k]} $ is skew symmetric 
 in $t_1,t_2, \dots,t_k$ and 
in $t_{k+1},t_{k+2},\dots,t_{2n}$, respectively.
\end{enumerate}
\end{lemma}

\begin{proof}
Recalling the definition of the Pfaffian, we have
\[
\frac{\partial}{\partial t_{k+1}} \pf \mathbb{L}^{[k]} 
= \sum_\eta \epsilon(\eta) \frac{\partial}{\partial t_{k+1}}
\prod_{i=1}^{n}
\mathbb{L}^{[k]}_{\eta(2i-1)\eta(2i)}.
\]
For each $i<j$, we see that:
\begin{align*}
 \text{if $j=k+1$, then} \qquad 
&
\frac{\partial}{\partial t_{k+1}} \mathbb{L}^{[k]}_{i,k+1}=\frac{\partial}{\partial t_{k+1}} \mathbb{K}_{12}(t_i,t_{k+1})
= \mathbb{K}_{11}(t_i,t_{k+1}) =\mathbb{L}^{[k+1]}_{i, k+1}; \\
\text{if $i=k+1$, then } \qquad
&\frac{\partial}{\partial t_{k+1}} \mathbb{L}^{[k]}_{k+1,j}=
\frac{\partial}{\partial t_{k+1}} \mathbb{K}_{22}(t_{k+1},t_{j})
= \mathbb{K}_{12}(t_{k+1},t_{j}) =\mathbb{L}^{[k+1]}_{k+1,j}; \\
\text{if $i,j \not=k+1$, then} \qquad
&\frac{\partial}{\partial t_{k+1}} \mathbb{L}^{[k]}_{i,j}=0.
\end{align*}
Hence we obtain 
\[
\frac{\partial}{\partial t_{k+1}} \pf \mathbb{L}^{[k]} 
= \sum_\eta \epsilon(\eta) 
\prod_{i=1}^{n}
\mathbb{L}^{[k+1]}_{\eta(2i-1)\eta(2i)}
= \pf \mathbb{L}^{[k+1]},
\]
which is the first claim.

Pfaffians are skew symmetric with respect to the change of 
the order of rows/columns, i.e., 
$\pf (a_{\eta(i) \eta(j)}) = \epsilon(\eta) \pf A$
for any $2n \times 2n$ skew symmetric matrix $A=(a_{ij})$ and 
a permutation $\eta$ on $\{1,2,\dots,2n\}$.
Hence the second claim follows from the definition of $\mathbb{L}^{[k]}$.
\end{proof}

Put
\[
\mf{X}_{2n}^< = \{(t_1,\dots,t_{2n}) \in \mf{X}_{2n} \ | \ t_1< \cdots <t_{2n} \}.
\]

\begin{lemma} \label{lem:sgn-limit}
\[
\lim_{\begin{subarray}{c}t_{2n} \to 1 \\ (t_1,\dots,t_{2n})\in \mf{X}_{2n}^< \end{subarray}
} E[\sgn f(t_1) \cdots \sgn f(t_{2n})]=0.
\]
\end{lemma}

\begin{proof}
If we put $X(t)= \sqrt{1-t^2} f(t)$, then
\[
E[\sgn f(t_1) \cdots \sgn f(t_{2n})] = E[\sgn X(t_1) \cdots \sgn X(t_{2n})]
\] 
and 
$E[X(t_i) X(t_j)]=c(t_i,t_j)$ with $c(s,t)$ in \eqref{eq:definition-c}. 
Furthermore, since $\lim_{t_{2n} \to 1} c(t_i,t_{2n})= \delta_{i,2n}$,
the random variable $X(t_{2n})$ converges in distribution 
to a standard Gaussian variable independent of 
other $X_i$ $(i <2n)$, which means that
$E[\sgn X(t_1) \cdots \sgn X(t_{2n})] \to 0$.
\end{proof}

\begin{lemma} \label{lem:Z-limit}
For each $k=0,1,2\dots,2n-1$,
\[
\lim_{\begin{subarray}{c}t_{2n} \to 1 \\ 
(t_1,\dots,t_{2n})\in \mf{X}_{2n}^< \end{subarray}
} 
 \pf \mathbb{L}^{[k]}=0.
\]
\end{lemma}

\begin{proof}
Taking the limit $t_{2n} \to 1$,
each entry in the last row and column of $\mathbb{L}^{[k]}$ converges to zero, and 
thus so does $\pf \mathbb{L}^{[k]}$.
\end{proof}

\begin{lemma} \label{lem:differential-pfaffian}
Let $(t_1,\dots,t_{2n}) \in
\mf{X}_{2n}^<$.
For each $k=0,1,\dots,2n$,
\begin{equation} \label{eq:differential-pfaffian}
\frac{\partial^k}{\partial t_1 \cdots \partial t_k}
E[\sgn f(t_1) \cdots \sgn f(t_{2n})]
=\ABS{\frac{2}{\pi}}^n \pf \mathbb{L}^{[k]}.
\end{equation}
\end{lemma}

\begin{proof}
Consider the function $Z^{[k]}$ on $\mf{X}_{2n}$
defined by 
\begin{align*}
\lefteqn{Z^{[k]}(t_1,\dots,t_{2n})} \\
&= \frac{\partial^k}{\partial t_1 \cdots \partial t_k}
E[\sgn f(t_1) \cdots \sgn f(t_{2n})] 
-\ABS{\frac{2}{\pi}}^n  \prod_{1 \le i<j \le 2n} 
\sgn(t_j-t_i) \cdot \pf \mathbb{L}^{[k]}.
\end{align*}
Since $ \mathbb{L}^{[2n]}=\ABS{ \mathbb{K}_{11}(t_i,t_j) }_{i,j=1}^{2n}$,
Proposition \ref{prop:differential} implies that
$Z^{[2n]} \equiv 0$ on $\mf{X}_{2n}$.
Let $k \in \{0,1,\dots,2n-1\}$ and 
suppose that $Z^{[k+1]} \equiv 0$ on $\mf{X}_{2n}^<$.
Our goal is to prove $Z^{[k]} \equiv 0$ on $\mf{X}_{2n}^<$.

From the first statement of Lemma \ref{lem:Z-properties},
$\frac{\partial}{\partial t_{k+1}} Z^{[k]}(t_1,\dots,t_{2n})
= Z^{[k+1]}(t_1,\dots,t_{2n})$,
and hence our assumption implies
$\frac{\partial}{\partial t_{k+1}} Z^{[k]}(t_1,\dots,t_{2n})
=0$.
Therefore  $Z^{[k]}$
is independent of $t_{k+1}$.
From the second statement of Lemma \ref{lem:Z-properties},
$Z^{[k]}$ is symmetric in $t_{k+1},\dots,t_{2 n}$, 
and therefore $Z^{[k]}$ is also  independent of $t_{2n}$.
However, 
\[
\lim_{\begin{subarray}{c}t_{2n} \to 1 \\ (t_1,\dots,t_{2n})\in \mf{X}_{2n}^< \end{subarray}
}  Z^{[k]} =0
\]
by Lemmas \ref{lem:sgn-limit} and \ref{lem:Z-limit}.
Hence, $Z^{[k]}$   must be identically zero on $\mf{X}_{2n}^<$.
\end{proof}

\subsection{Proof of Theorem \ref{thm:sgn}}

\begin{proof}[Proof of Theorem \ref{thm:sgn}]
Lemma \ref{lem:differential-pfaffian} for $k=0$
implies 
\[
E[\sgn f(t_1) \cdots \sgn f(t_{2n})]
=\ABS{\frac{2}{\pi}}^n 
\pf (\mathbb{K}_{22}(t_i,t_j))_{i,j=1}^{2n}
\]
for $t_1< \cdots <t_{2n}$.
Since $\pf (\mathbb{K}_{22}(t_i,t_j))_{i,j=1}^{2n}$
is skew symmetric in $t_1,\dots,t_{2n}$, 
\[
\prod_{1 \le i< j \le 2n} 
\sgn(t_j-t_i) \cdot
\pf (\mathbb{K}_{22}(t_i,t_j))_{i,j=1}^{2n}
\]
is symmetric and coincides with 
$E[\sgn f(t_1) \cdots \sgn f(t_{2n})]$ on $\mf{X}_{2n}$.
Thus we have obtained Theorem \ref{thm:sgn}. 
\end{proof}

The following corollary is a consequence of Theorem \ref{thm:sgn}.

\begin{corollary} \label{cor:differential-pfaffian}
For $(t_1,\dots,t_{2n}) \in \mf{X}_{2n}$ and 
a subset $I=\{i_1<i_2< \cdots<i_k\}$ in $\{1,2,\dots,2n\}$,
\[
\frac{\partial^k}{\partial t_{i_1} \cdots \partial t_{i_k}}
E[\sgn f(t_1) \cdots \sgn f(t_{2n})]=
\ABS{\frac{2}{\pi}}^n \prod_{1 \le i<j \le 2n} \sgn(t_j-t_i) \cdot 
\pf \mathbb{L}^I,
\]
where $\mathbb{L}^I$ is defined in \eqref{eq:def-L}.
\end{corollary}

\begin{proof}
Observe that $\frac{\partial^k}{\partial t_{i_1} \cdots \partial t_{i_k}}
\pf \mathbb{L}^\emptyset = \pf \mathbb{L}^I$.
\end{proof}

Note that 
if $t_r=t_s$ for some $r \not=s$ then 
the both sides in the equation of the corollary vanish.

\subsection{Proof of Theorem \ref{thm:absolute}}

\begin{lemma} \label{lem:absolute-diffenretial-pfaffian}
For $(t_1,\dots,t_n) \in \mf{X}_n$,
\begin{align*}
&\lim_{s_1 \to t_1+0} \cdots \lim_{s_n \to t_n+0}
\frac{\partial^n}{\partial t_1 \cdots \partial t_n}
E[\sgn f(t_1) \cdots \sgn f(t_n) \sgn f(s_1) \cdots \sgn f(s_n)] \\
&= \ABS{\frac{2}{\pi}}^n \pf (\mathbb{K}(t_i,t_j))_{i,j=1}^n.
\end{align*}
\end{lemma}

\begin{proof}
Take $-1<t_1< s_1 < t_2<s_2< \cdots <t_n<s_n<1$.
Corollary \ref{cor:differential-pfaffian} with 
$I=\{1,3,5,\dots,2n-1\}$ and with the replacement
$(t_1,\dots,t_{2n})$ by $(t_1,s_1,\dots,t_n,s_n)$ gives 
\[
\frac{\partial^n}{\partial t_1 \cdots \partial t_n}
E[\sgn f(t_1) \cdots \sgn f(t_n) \sgn f(s_1) \cdots \sgn f(s_n)] 
= \ABS{\frac{2}{\pi}}^n  
\pf (\mathfrak{K}_{ij})_{i,j=1}^n,
\]
where
\[
\mathfrak{K}_{ij}= 
\begin{pmatrix} 
\mathbb{K}_{11}(t_{i},t_j) & \mathbb{K}_{12}(t_i,s_j) \\
\mathbb{K}_{21}(s_i,t_j) & \mathbb{K}_{22}(s_i,s_j)
\end{pmatrix}.
\]
Taking the limit $s_1 \to t_1, \dots, s_n \to t_n$,
it converges to
$ \ABS{\frac{2}{\pi}}^n \pf (\mathbb{K}(t_i,t_j))_{i,j=1}^n$
for $t_1< \cdots <t_{n}$.
From the symmetry for $t_1,\dots,t_n$, 
the achieved result holds true for every $(t_1,\dots,t_n) \in \mf{X}_n$.
\end{proof}

\begin{proof}[Proof of Theorem \ref{thm:absolute}]

We use  Lemma \ref{lem:key-lem} with 
$m=n$.
The identity in the lemma holds true
for $-1<t_1<s_1 <t_2<s_2< \cdots<t_n <s_n<1$.
Note that $\prod_{i=1}^n \prod_{j=1}^n \sgn(s_j-t_i)=(-1)^{n(n-1)/2}$.
Taking the limit 
$s_1 \to t_1, \dots, s_n \to t_n$,
\begin{align*}
&\lim_{s_1 \to t_1+0} \cdots \lim_{s_n \to t_n+0}
\frac{\partial^n}{\partial t_1 \cdots \partial t_n}
E[\sgn f(t_1) \cdots \sgn f(t_n) \sgn f(s_1) \cdots \sgn f(s_n)] \\
&=  \ABS{\frac{2}{\pi}}^{n/2}
(\det \Sigma(\bm{t}))^{1/2} E[ |f(t_1) \cdots f(t_n)|]
\end{align*}
for $-1 <t_1< \cdots <t_n<1$.
From the symmetry for $t_1,\dots,t_n$,
the above equation holds true for every $(t_1,\dots,t_n) \in \mf{X}_n$.
Combining this fact with Lemma \ref{lem:absolute-diffenretial-pfaffian},
we obtain Theorem \ref{thm:absolute}.
\end{proof}

\section{Proofs of Theorem \ref{thm:correlation}, 
Corollary~\ref{2point} and Corollary~\ref{cor:variance}}
\label{sec:correlation}

\subsection{Proof of Theorem \ref{thm:correlation}}

Hammersley's formula \cite{Ham} describes correlation functions of zeros 
of random polynomials, which was observed by Hammersley 
and it is extended to Gaussian analytic functions 
as Corollary 3.4.2 in \cite{HKPV}. 
The following lemma is a real 
version of Hammersley's formula for 
correlation functions of Gaussian analytic functions. 

\begin{lemma} \label{lem:Hammersley}
Let $X(t)$ be a random power series with independent real Gaussian coefficients 
defined on an interval $(-1, 1)$ with covariance kernel $K$. 
If $\det K(\bm{t}) = \det (K(t_i,t_j))_{i,j=1}^n$ does not vanish anywhere on $\mf{X}_n$, 
then the $n$-point correlation function for real zeros of $f$
exists and is given by 
\[
\rho_n(t_1,\dots,t_n) = \frac{E[|X'(t_1) \cdots X'(t_n)| \ | \ 
X(t_1)= \cdots =X(t_n)=0]}{(2\pi)^{n/2} \sqrt{\det K(\bm{t})}}
\]
for $\bm{t}=(t_1,\dots,t_n) \in \mf{X}_n$. 
\end{lemma}
\begin{proof}
This can be proved in almost the same way as in the proof of (3.4.1) in 
Corollary 3.4.2 in \cite{HKPV}. 
The only difference is that the exponent of $|X'(t_1) \cdots X'(t_n)|$
 is $1$ in the case of real Gaussian coefficients 
instead of $2$ in the complex case. 
This is due to the fact that 
the Jacobian determinant of $F(\bm{t}) = (X(t_1), \dots, X(t_n))$ 
is equal to $|X'(t_1) \cdots X'(t_n)|$ when $X$ is a real-valued 
differentiable function 
while $|X'(t_1) \cdots X'(t_n)|^2$ when $X$ is complex-valued.  
\end{proof}

\begin{proof}[Proof of Theorem \ref{thm:correlation}]
From Lemma \ref{lem:conditional} and \eqref{eq:product-of-q} we have
\begin{align*}
E[|f'(t_1) \cdots f'(t_n)| \ | \ f(t_1)= \cdots =f(t_n)=0]
&=  |q_1(\bm{t}) \cdots q_n(\bm{t})| E[|f(t_1) \cdots f(t_n)|] \\
&=  \det \Sigma(\bm{t}) \cdot E[|f(t_1) \cdots f(t_n)|],
\end{align*}
and it follows from Lemma \ref{lem:Hammersley} that
\[
\rho_n(t_1,\dots,t_n)= (2\pi)^{-n/2} (\det \Sigma(\bm{t}))^{1/2} E[|f(t_1) \cdots f(t_n)|].
\]
Hence Theorem \ref{thm:correlation} follows from Theorem \ref{thm:absolute}.
\end{proof}

\subsection{Proof of Corollaries~\ref{2point} and \ref{cor:variance}}

\begin{proof}[Proof of Corollary~\ref{2point}] 
From (\ref{2correlation}) we observe that 
\[
R(s,t) =  1 + |\mu(s,t)| c(s,t) \arcsin c(s,t) - c(s,t)^2,  
\]
and so $R(s,s)=0$ and $R(s,\pm 1) =1$. 
A simple calculation yields 
\begin{align*}
 \frac{\partial}{\partial t} |\mu(s,t)| 
&= \frac{\sgn(t-s)}{1-t^2} c(s,t)^2, \qquad 
 \frac{\partial}{\partial t} c(s,t)
= - \frac{\sgn(t-s)}{1-t^2} |\mu(s,t)| c(s,t), \\
 \frac{\partial}{\partial t} \arcsin c(s,t) 
&= - \frac{\sgn(t-s)}{1-t^2} c(s,t)
\end{align*}
and hence we obtain 
\[
 \frac{\partial}{\partial t} R(s,t)
= \frac{\sgn(t-s)}{1-t^2}
c(s,t) g(s,t), 
\] 
where 
\[
g(s,t) := \{c(s,t)^2 \arcsin c(s,t) - |\mu(s,t)|^2 \arcsin c(s,t) 
+ c(s,t)|\mu(s,t)|\}. 
\]
Since $g(s,\pm 1)=0$ and 
\begin{align*}
\frac{\partial}{\partial t} 
g(s,t) = \frac{-4 \sgn(t-s)}{1-t^2} |\mu(s,t)| c(s,t)^2 \arcsin c(s,t),  
\end{align*}
we have $g(s,t) \ge 0$. This implies the claim.  
\end{proof}

\begin{proof}[Proof of Corollary~\ref{cor:variance}]
The first equality immediately follows from $E N_r = \int_{-r}^r
 \rho_1(s) ds$. Recall that 
\[
 \var N_r 
= \int_{-r}^r \int_{-r}^r \rho_2(s,t) ds dt 
+ \int_{-r}^r \rho_1(s) ds - \left(\int_{-r}^r \rho_1(s) ds\right)^2. 
\]
We recall the $2$-correlation function 
 \[
 \rho_2(s,t) = \pi^{-2}\{\mathbb{K}_{12}(s,s) \mathbb{K}_{12}(t,t)
- \mathbb{K}_{11}(s,t) \mathbb{K}_{22}(s,t)
+ \mathbb{K}_{12}(s,t) \mathbb{K}_{21}(s,t)\}
\]
from (\ref{2correlation}). 
Taking the discontinuity of $\mathbb{K}_{22}(s,t)$ at $s=t$ into account
and using integration by parts together with (\ref{derivative}), 
we have 
\begin{align*}
& \int_{-r}^r \mathbb{K}_{11}(s,t) \mathbb{K}_{22}(s,t) dt \\
&= \int_{-r}^s \frac{\partial^2 \mathbb{K}_{22}}{\partial s \partial t} 
(s,t) \mathbb{K}_{22}(s,t) dt
+ \int_s^r \frac{\partial^2 \mathbb{K}_{22}}{\partial s \partial t}(s,t) 
\mathbb{K}_{22}(s,t) dt \\
&= 
[\frac{\partial \mathbb{K}_{22}}{\partial s}
(s,t) \mathbb{K}_{22}(s,t)]_{t=-r}^{s-0}  
+ [\frac{\partial \mathbb{K}_{22}}{\partial s}(s,t)
 \mathbb{K}_{22}(s,t)]_{t=s+0}^r 
- \int_{-r}^r \mathbb{K}_{12}(s,t) \mathbb{K}_{21}(s,t) dt \\
&= \{ 
\mathbb{K}_{12}(s,r) \mathbb{K}_{22}(s,r) - 
\mathbb{K}_{12}(s,-r) \mathbb{K}_{22}(s,-r) 
- \pi \mathbb{K}_{12}(s,s)\}
- \int_{-r}^r \mathbb{K}_{12}(s,t) \mathbb{K}_{21}(s,t) dt. 
\end{align*}
It is easy to see that 
\begin{align*}
\lefteqn{\int_{-r}^r 
ds \{\mathbb{K}_{12}(s,r) \mathbb{K}_{22}(s,r) - 
\mathbb{K}_{12}(s,-r) \mathbb{K}_{22}(s,-r) \}} \\
&= \frac{1}{2} [\mathbb{K}_{22}(s,r)^2 - \mathbb{K}_{22}(s,-r)^2]_{-r}^r 
= \mathbb{K}_{22}(r,r)^2 - \mathbb{K}_{22}(r,-r)^2 = O(1). 
\end{align*}
Hence, we obtain 
\begin{align*}
 \var N_r 
=& 2 \pi^{-2} \left(\pi
\int_{-r}^r \mathbb{K}_{12}(s,s) ds +  
\int_{-r}^r \int_{-r}^r \mathbb{K}_{12}(s,t) \mathbb{K}_{21}(s,t) ds dt
\right) + O(1)\\ 
=& 2 \pi^{-2} \left(\pi \log \frac{1+r}{1-r}
- 2\log \frac{1+r^2}{1-r^2} \right) + O(1). 
\end{align*}
This implies the assertion. 
\end{proof}

\section{Proof of Theorem \ref{thm:c-correlation}}

\subsection{Complex-valued Gaussian processes}

Let $X = \{X(\la)\}_{\la \in \La}$ be a centered complex-valued 
Gaussian process in the sense that the real and imaginary parts form 
centered real Gaussian processes.
Here 
we say that a complex-valued Gaussian process is a \emph{complex
Gaussian process} if the real and imaginary parts are mutually
independent and have the same variance.  

For a complex-valued Gaussian process $X$, we use three $2 \times 2$ matrices
\begin{align*}
 \mathbb{M}_X(\la,\mu)  
=&\begin{pmatrix} 
E[X(\la) \overline{X(\mu)}] & E[X(\la) X(\mu)] \\
E[\overline{X(\la) X(\mu)}] & E[\overline{X(\la)} X(\mu)]
\end{pmatrix}, \\ 
 \widehat{\mathbb{M}}_X(\la,\mu) =& 
\begin{pmatrix} 
E[X(\la) X(\mu)] & E[X(\la) \overline{X(\mu)}] \\
E[\overline{X(\la)} X(\mu)] & E[\overline{X(\la) X(\mu)}]
\end{pmatrix} = \mathbb{M}_X(\lambda,\mu)
\begin{pmatrix} 0 & 1 \\ 1 & 0 \end{pmatrix}, \\
\widetilde{\mathbb{M}}_X(\lambda,\mu) 
=& \begin{pmatrix}
E[\Re X(\la) \Re X(\mu)] & E[\Re X(\la) \Im X(\mu)] \\
E[\Im X(\la) \Re X(\mu)] & E[\Im X(\la) \Im X(\mu)]
\end{pmatrix}. 
\end{align*}
For $\la_1,\la_2,\dots, \la_n \in \La$, the matrix 
$(\mathbb{M}_X(\la_i,\la_j))_{i,j=1}^n$ is Hermitian,  
$(\widehat{\mathbb{M}}_X(\la_i,\la_j))_{i,j=1}^n$ is complex symmetric,
and  $(\widetilde{\mathbb{M}}_X(\la_i,\la_j))_{i,j=1}^n$ is real symmetric.
The real Gaussian vector 
\[
(\Re X(\la_1), \Im X(\la_1), \dots, \Re X(\la_n), \Im X(\la_n))
\]
has the covariance matrix
$(\widetilde{\mathbb{M}}_X(\la_i,\la_j))_{i,j=1}^n$.
We can see that 
\begin{equation}
\widetilde{\mathbb{M}}_X(\lambda,\mu) 
= \frac{1}{4} U \mathbb{M}_X(\la, \mu) U^*, 
\quad 
U = 
\begin{pmatrix}
1 & 1 \\
- \ii & \ii
\end{pmatrix}. 
\label{eq:Mtilde}
\end{equation}
A (centered) complex-valued Gaussian process is uniquely determined 
by $\mathbb{M}_X$ or $\widehat{\mathbb{M}}_X$.

\begin{lemma} For $\la_1,\dots, \la_n \in \La$, 
 \begin{equation}
  \mathbb{E}[|X(\la_1) \cdots X(\la_n)|^2]
= \hf (\widehat{\mathbb{M}}_X(\la_i,\la_j))_{i,j=1}^n. 
\label{eq:hafnian}
\end{equation}
\end{lemma}
\begin{proof}
Let $Y_a(\la) = \Re X(\la) + a \Im X(\la)$ for $a \in \mathbb{R}$. 
It follows from the Wick formula (\ref{wick}) that 
\begin{equation}
 \mathbb{E}[Y_a(\la_1) Y_b(\la_1) \cdots Y_a(\la_n) Y_b(\la_n) ]
= \hf (\mf{Y}_{ij})_{i,j=1}^n, 
\label{y-wick} 
\end{equation}
where
\[
\mf{Y}_{ij}=\begin{pmatrix}
 E[Y_a(\la_i) Y_a(\la_j)] &  E[Y_a(\la_i) Y_b(\la_j)] \\
 E[Y_b(\la_i) Y_a(\la_j)] &  E[Y_b(\la_i) Y_b(\la_j)]
\end{pmatrix}.
\]
By analytic continuation, the formula (\ref{y-wick}) still holds 
for $a,b \in \mathbb{C}$. Therefore, by setting $a=-b=\ii$, 
we obtain the result. 
\end{proof}

\subsection{Conditional expectations for complex cases}

Let $\mathbb{D}$ be the open unit disc $\mathbb{D}= \{ z\in \mathbb{C} \ | \ 
|z|<1\}$.
We naturally extend the definition $\mu$ in \eqref{eq:definition-sm} to $\mathbb{D}$:
\[
\mu(z,w)= \frac{z-w}{1-zw} \qquad (z,w \in \mathbb{D}).
\]
For $z_1,\dots,z_n \in \mathbb{D}$
and $i=1,2,\dots,2n$, 
we consider 
\[
q_i(\bm{z})=q_i(z_1,\dots,z_n,z_{n+1},\dots,z_{2n})
\]
defined in \eqref{eq:definition-q}
with $z_{j+n}:=\overline{z_j}$ $(j=1,2,\dots,n)$.

Recall that $\mathbb{D}_+$ is the upper half of the open unit disc:
$\mathbb{D}_+=\{z \in \mathbb{D} \ | \ \Im(z)>0\}$.
If $f$ is the Gaussian power series defined by \eqref{eq:definition-f},
then $\{f (z)\}_{z \in \mathbb{D}_+}$ is a complex-valued 
Gaussian process with 
\[
\mathbb{M}_f(z,w)= \begin{pmatrix} \frac{1}{1-z \overline{w}} & \frac{1}{1-zw} \\
\frac{1}{1- \overline{zw}} & \frac{1}{1-\overline{z}w}
\end{pmatrix}.
\]

\begin{lemma}
 For $\eta \in \mathbb{D}_+$, 
\begin{equation} \label{eq:complex_f_law}
 (f \ | \ f(\eta)=0) \inlaw \mu(\cdot, \eta)  \mu(\cdot, \bar{\eta}) f. 
\end{equation}
Moreover,  
\begin{equation}
 (f'(z_i), i=1,2,\dots,n \ | \ f(z_1) =\cdots = f(z_n)=0) 
\inlaw (q_i(\bm{z}) f(z_i), i=1,2,\dots, n). 
\label{eq:inlaw}
\end{equation}
\end{lemma}

\begin{proof}
The real Gaussian vector $(\Re f(z), \Im f(w))$, given 
$\Re f(\eta)= \Im f(\eta)=0$, has the covariance matrix 
\begin{align*}
\lefteqn{ 
\widetilde{\mathbb{M}}_f(z,w) -
\widetilde{\mathbb{M}}_f(z,\eta) \widetilde{\mathbb{M}}_f(\eta,\eta)^{-1}
\widetilde{\mathbb{M}}_f(\eta,w)} \\
&=\frac{1}{4} U
[\mathbb{M}_f(z,w) -
\mathbb{M}_f(z,\eta) \mathbb{M}_f(\eta,\eta)^{-1}
\mathbb{M}_f(\eta,w)] U^* 
\end{align*}
by \eqref{eq:Mtilde}.
A direct computation gives 
\[
\mathbb{M}_f(z,w) -
\mathbb{M}_f(z,\eta) \mathbb{M}_f(\eta,\eta)^{-1}
\mathbb{M}_f(\eta,w) = 
\mathbb{M}_{g_\eta}(z,w)
\]
with $g_\eta(z)= \mu(z,\eta)\mu(z,\overline{\eta}) f(z)$,
and we obtain \eqref{eq:complex_f_law}.
The remaining statement follows from \eqref{eq:complex_f_law}
in a manner similar to the proof of Lemma \ref{lem:conditional}.
\end{proof}

\subsection{Correlation functions for complex zeros}

We finally compute the correlation function 
$\rho_n^{\mathrm{c}}(z_1,\dots,z_n)$ for complex zeros of $f$.
Our starting point is the following Hammersley's formula (complex version),
see \cite{HKPV} and compare with Lemma \ref{lem:Hammersley}:
\begin{equation} \label{eq:Hammersley_complex}
 \rho_n^{\mathrm{c}}(z_1,\dots,z_n) = 
\frac{E[|f'(z_1) \cdots f'(z_n)|^2 \ | \ f(z_1) = \cdots =
 f(z_n)=0]}{(2\pi)^n \sqrt{\det(\widetilde{\mathbb{M}}_f(z_i,z_j))_{i,j=1}^n}}. 
\end{equation}
Note that $(2\pi)^{-n} [\det(\widetilde{\mathbb{M}}_f(z_i,z_j))
]^{-1/2}$ is the density 
of the real Gaussian vector 
\[
(\Re f(z_1), \Im f(z_1), \dots, \Re f(z_n), \Im f(z_n))
\]
at $(0,0,\dots,0,0)$.

\begin{proposition}\label{prop:complex-corr}
Let 
\[
M(\bm{z})= \ABS{\frac{1}{1-z_i \bar{z_j}} }_{i,j=1}^{2n}
\qquad \text{and} \qquad 
\hat{M}(\bm{z})= \ABS{\frac{1}{1-z_i z_j} }_{i,j=1}^{2n}.
\]
Then
\[
\rho_n^{\mathrm{c}}(z_1,\dots,z_n) 
= \frac{(-1)^n \det \hat{M}(\bm{z}) \cdot \hf \hat{M}(\bm{z})}{\pi^n 
\sqrt{\det M(\bm{z})}}. 
\] 
\end{proposition}

\begin{proof}
Let us compute the numerator on \eqref{eq:Hammersley_complex}.
Equation \eqref{eq:inlaw} gives
\[
E[|f'(z_1) \cdots f'(z_n)|^2 \ | \ f(z_1) = \cdots = f(z_n)=0] 
= |q_1(\bm{z})  \cdots q_n(\bm{z})|^2 E[|f(z_1) \cdots f(z_n)|^2].
\]
Here, since $\overline{q_j(\bm{z})}=q_{j+n}(\bm{z})$ for $j=1,2,\dots,n$,
it follows from \eqref{eq:product-of-q} that
\[
|q_1(\bm{z})  \cdots q_n(\bm{z})|^2 =\prod_{i=1}^{2n}
q_i(\bm{z})= (-1)^n \det \hat{M}(\bm{z}).
\]
Furthermore, from \eqref{eq:hafnian} we have
\[
E[|f(z_1) \cdots f(z_n)|^2] = \hf (\widehat{\mathbb{M}}_f(z_i,z_j))_{i,j=1}^n
= \hf \hat{M}(\bm{z}).
\]

On the other hand, the denominator on \eqref{eq:Hammersley_complex}
is computed by using \eqref{eq:Mtilde}:
\[
 \det(\widetilde{\mathbb{M}}_f(z_i,z_j))_{i,j=1}^n= 4^{-n} \det M(\bm{z}). 
\]
Consequently, we obtain the result from \eqref{eq:Hammersley_complex}. 
\end{proof} 

\begin{proof}[Proof of Theorem~\ref{thm:c-correlation}]
By the Cauchy determinant formula \eqref{eq:Cauchy},  
\begin{align*}
\det \hat{M}(\bm{z}) 
&= \prod_{i=1}^{2n} \frac{1}{1-z_i^2} 
\prod_{1 \le i < j \le 2n} \ABS{\frac{z_i-z_j}{1-z_i z_j}}^2, \\
\sqrt{\det M(\bm{z})}
&= \prod_{i=1}^{2n} \frac{1}{\sqrt{1-|z_i|^2}} \cdot  
\prod_{1 \le i < j \le 2n} \left| \frac{z_i-z_j}{1-z_i \overline{z_j}} \right|. 
\end{align*}
By noting that $z_{i+n} = \bar{z}_i \ (i=1,2,\dots,n)$,
We can see that 
\begin{align*}
\frac{\det \hat{M}(\bm{z})}{\sqrt{\det M(\bm{z})}}
&= (-1)^{n(n-1)/2} \prod_{i=1}^{n} \frac{1}{|1-z_i^2|} 
\cdot 
\prod_{i=1}^{n} \frac{z_i -\bar{z}_i}{|z_i - \bar{z}_i|}
\left(\prod_{1 \le i < j \le 2n} \frac{z_i-z_j}{1-z_iz_j} \right).
\end{align*}
Since $\frac{z -\bar{z}}{|z - \bar{z}|} = \ii$ for 
$\Im z >0$, from Proposition~\ref{prop:complex-corr} 
and Pfaffian-Hafnian identity \eqref{eq:pfaffian-hafnian} 
we see that 
\begin{align*}
\rho_n^{\mathrm{c}}(z_1,\dots,z_n) 
&= \frac{(-1)^{n(n-1)/2} }{(\pi\ii)^n}  
\prod_{i=1}^{n} \frac{1}{|1-z_i^2|} \cdot
\pf \ABS{\frac{z_i-z_j}{(1-z_i z_j)^2}}_{i,j=1}^{2n}.
\end{align*} 
By changing rows and columns in the Pfaffian, we finally obtain 
\[
\rho_n^{\mathrm{c}}(z_1,\dots,z_n) 
= \frac{1}{(\pi \ii)^n} 
\prod_{i=1}^n \frac{1}{|1-z_i^2|}
\cdot 
\pf\ABS{\mathbb{K}^{\mathrm{c}}(z_i,z_j)}_{i,j=1}^n.  
\]
\end{proof}

\bigskip

\section*{Acknowledgments}

The first author (SM)'s work was 
supported by JSPS Grant-in-Aid for Young Scientists (B) 22740060. 
The second author (TS)'s work was supported in part 
by JSPS Grant-in-Aid for Scientific Research (B) 22340020. 
S.M. would like to thank Yuzuru Inahama for his helpful conversations.
The authors appreciate referee's kind comments and suggestions.


\bigskip

\noindent
\textsc{Sho Matsumoto} \\
Graduate School of Mathematics, Nagoya University, Nagoya, 464-8602, Japan. \\
\verb|sho-matsumoto@math.nagoya-u.ac.jp|

\medskip

\noindent
\textsc{Tomoyuki Shirai} \\
Institute of Mathematics for Industry, Kyushu University, 
Fukuoka 819-0395, Japan. \\
\verb|shirai@imi.kyushu-u.ac.jp|

\end{document}